\documentclass[11pt]{article}
\usepackage{amsmath, amsfonts, amsthm, amssymb, color}
\usepackage{graphicx}
\usepackage{float}
\usepackage{verbatim}
\usepackage{bm}
\allowdisplaybreaks

\numberwithin{equation}{section}

\newtheorem{lemma}{Lemma}[section]

\newtheorem{prop}[lemma]{Proposition}
\newtheorem{thm}[lemma]{Theorem}

\theoremstyle{definition}

\newtheorem{example}[lemma]{Example}
\newtheorem{conjecture}[lemma]{Conjecture}

\theoremstyle{remark}

\def\p{\mathbf{p}}

\def\A{\mathcal{A}}

\def\I{\mathcal{I}}
\def\om{\bm{\omega}}
\def\F{\mathcal{F}}
\numberwithin{equation}{section} \numberwithin{table}{section}

\title{On normal numbers and self-similar measures}
\author{Simon Baker\\ \\
\emph{School of Mathematics,} \\ \emph{University of Birmingham,} \\ \emph{Birmingham,  B15 2TT, UK.} \\ Email: simonbaker412@gmail.com\\}

\date{\today}
\begin{document}
\maketitle

\begin{abstract}
In this paper we prove that if $\{\varphi_i(x)=\lambda x+t_i\}$ is an equicontractive iterated function system and $b$ is a positive integer satisfying $\frac{\log b}{\log |\lambda|}\notin\mathbb{Q},$ then almost every $x$ is normal in base $b$ for any non-atomic self-similar measure of $\{\varphi_i\}$.\\

 


\end{abstract}

\section{Introduction}
Let $b\geq 2$ be an integer. A real number $x$ is said to be normal in base $b$ if the sequence $(b^nx)_{n=1}^{\infty}$ is uniformly distributed modulo one. For a real number $x,$ being normal in base $b$ indicates that the base $b$ expansion of $x$ behaves like a sequence of uniformly distributed i.i.d random variables. The study of normal numbers was pioneered by Borel in \cite{Bor}. He proved that Lebesgue almost every $x$ is normal in base $b$ for any integer $b\geq 2$. Since these beginnings, the study of normal numbers has developed into an important and active branch of mathematics. It has significant connections to Ergodic Theory, Fractal Geometry, and Number Theory. We refer the reader to the books \cite{Bug} and \cite{KN} for a more detailed introduction to this subject.

Despite the result of Borel mentioned above, it is often a challenging problem to demonstrate that a real number is normal in a given base $b$. Indeed there are relatively few explicit examples of real numbers that are normal in a base $b$ (see \cite{Bug}). Nevertheless it is reasonable to expect that a real number will be normal in base $b$ if it is defined in a manner that is independent from $b$. This reasoning leads to natural conjectures which state that well known constants like $\pi$ and $e$ are normal in any base. These conjectures are extremely challenging and it is necessary to readjust our expectations. Instead of studying specific real numbers, one can study Borel probability measures on $\mathbb{R}$. If a Borel probability measure $\mu$ is defined in a manner that is independent from $b$, then it is reasonable to expect that $\mu$ almost every $x$ will be normal in base $b$. In this paper we will pursue this line of research in the special case when $\mu$ is a self-similar measure of an iterated function system.

\subsection{Background and statement of results}
We call a map $\varphi:\mathbb{R}\to\mathbb{R}$ a contracting similarity if there exists $\lambda\in (-1,0)\cup (0,1)$ and $t\in\mathbb{R}$ such that $\varphi(x)=\lambda x +t$. We call a finite set of contracting similarities an iterated function system or IFS for short. A well known theorem due to Hutchinson \cite{Hut} states that for any IFS $\{\varphi_i\}_{i\in \I},$ there exists a unique non-empty compact set $X$ satisfying 
\begin{equation}
\label{selfsimilarset}
X=\bigcup_{i\in \I} \varphi_{i}(X).
\end{equation} We call $X$ the self-similar set of $\{\varphi_i\}_{i\in \I}$. Many well known fractal sets can be realised as the self-similar set of an iterated function system. For example the middle third Cantor set is the self-similar set of the IFS $\{\phi_{1}(x)=\frac{x}{3},\phi_{2}(x)=\frac{x+2}{3}\}.$ Given an IFS $\{\varphi_i\}_{i\in \I}$ and a probability vector $\p=(p_i)_{i\in \I}$, there exists a unique Borel probability measure $\mu_{\p}$ supported on $X$ satisfying $$\mu_{\p}=\sum_{i\in \I}p_i\cdot \varphi_{i}\mu_{\p}.$$ Here $\varphi_{i}\mu_{\p}$ denotes the pushforward of $\mu_{\p}$ by the similarity $\varphi_{i}$. We call $\mu_{\p}$ the self-similar measure corresponding to $\{\varphi_i\}_{i\in \I}$ and $\p$. We say that a self-similar measure $\mu_\p$ is fully supported if $p_{i}>0$ for all $i\in \I$. It is often the case that there is no loss of generality in assuming that a self-similar measure is fully supported. This is because if a self-similar measure is not fully supported, for each $i\in\I$ satisfying $p_i=0$ we can remove the corresponding similarity from the IFS and remove the corresponding entry from $\p$. This new IFS and probability vector yields the same self-similar measure. However it is now fully supported with respect to the new probability vector. 

Many important properties of a self-similar set depend upon the contraction ratios of the similarities in the IFS, i.e. those $\lambda_i\in (-1,0)\cup (0,1)$ such that $\varphi_i(x)=\lambda_{i}x+t_{i}$. With our previous discussion in mind, it is reasonable to expect that the arithmetic properties of the contraction ratios may influence the existence of normal numbers within the self-similar set. The following conjecture is natural and follows from these considerations.

\begin{conjecture}
	\label{Main conjecture}
Let $\{\varphi_i(x)=\lambda_i x+t_i\}_{i\in \I}$ be an IFS and $b\geq 2$ be an integer. Suppose that $\frac{\log b}{\log |\lambda_i|}\notin\mathbb{Q}$ for some $i\in \I$. Then almost every $x$ is normal in base $b$ for any non-atomic fully supported self-similar measure of $\{\varphi_i\}_{i\in \I}$.
\end{conjecture}
Clearly Conjecture \ref{Main conjecture} is false without the assumption $\frac{\log b}{\log |\lambda_i|}\notin\mathbb{Q}$ for some $i\in \I$. Take for example the middle third Cantor set. It contains no real numbers that are normal in base $3$. The existence of the digit $i\in \I$ for which $\frac{\log b}{\log |\lambda_i|}\notin\mathbb{Q}$ should be interpreted as a mechanism ensuring that the self-similar set $X$ has no arithmetic structure relating to the base $b$. This lack of structure means that the non-atomic fully supported self-similar measures are in a sense independent from $b$. As such it is reasonable to expect that these measures would give full mass to the set of real numbers that are normal in base $b$. 

The first instances of a special case of Conjecture \ref{Main conjecture} being proved can be found in the papers of Cassels and Schmidt \cite{Cas,Sch}. They considered the IFS $\{\phi_{1}(x)=\frac{x}{3},\phi_{2}(x)=\frac{x+2}{3}\}$ and the probability vector $\p=(1/2,1/2)$.  They proved that with respect to the corresponding self-similar measure, almost every $x$ is normal in base $b$ if $b$ is not a power of three. By proving this result, these authors answered in the affirmative a question of Steinhaus on whether there exists a real number that is normal in base $b$ for infinitely many $b$ but not all $b$. Hochman and Shmerkin proved Conjecture \ref{Main conjecture} under the additional assumption that the IFS satisfies a certain separation condition \cite{HocShm}. Important progress towards a proof of Conjecture \ref{Main conjecture} was made in a recent paper by Algom et al \cite{ARHW}. In this paper the authors proved that if a self-similar measure $\mu_{\p}$ is a Rajchman measure, that is its Fourier transform converges to zero, then $\mu_{\p}$ almost every $x$ is normal in base $b$ for any $b\geq 2$. By combining this result with existing theorems on when self-similar measures are Rajchman measures, we can conclude a number of special cases of Conjecture \ref{Main conjecture}. In particular, by a result of Li and Sahlsten \cite{LS}, it follows that Conjecture \ref{Main conjecture} is true under the additional assumption that there exists $i,i'\in I$ such that $\frac{\log |\lambda_i|}{\log |\lambda_{i'}|} \notin \mathbb{Q}.$ This leaves open the case when all of the contraction ratios are integer powers of some parameter $\lambda$. This case was studied by Br\'{e}mont in \cite{Bre}. He proved that for such an IFS, if $\mu_{\p}$ is a non-atomic self similar measure and it is not a Rajchman measure, then the parameter $\lambda$ is the reciprocal of a Pisot number and the underlying IFS can be conjugated to an IFS with translation parameters satisfying $t_i\in\mathbb{Q}(\lambda)$ for all $i\in \I$. Recall that a real number is said to be a Pisot number if it is an algebraic integer greater than $1$ whose Galois conjugates all have modulus strictly less than $1$. The golden mean is an example of a Pisot number. Combining this result of Br\'{e}mont with the aforementioned result of Algom et al, it follows that Conjecture \ref{Main conjecture} is true if the contraction ratios are not all integer powers of a common Pisot number. In another recent paper, Hochman gave an alternative proof of a well known theorem due to Host \cite{Hoc}. At the end of this paper, Hochman commented that his method could be extended to prove Conjecture \ref{Main conjecture} under the assumption that the IFS satisfies the strong separation condition, i.e. $\varphi_{i}(X)\cap \varphi_{i'}(X)=\emptyset$ for $i,i'\in \I$ such that $i\neq i'$.

In this paper we prove Conjecture \ref{Main conjecture} under the assumption that the IFS is equicontractive, i.e. there exists $\lambda\in(-1,0)\cup(0,1)$ such that $\lambda_{i}=\lambda$ for all $i\in \I$. We emphasise that this result allows for $\lambda$ to be the reciprocal of a Pisot number, and can handle self-similar measures that are not Rajchman measures, see Example \ref{Bernoulliconvolution}. Our proof is independent from the work of Algom et al and Br\'{e}mont, and treats both the Rajchman and non-Rajchman cases simultaneously.


\begin{thm}
	\label{Main theorem}
	Let $\{\varphi_i(x)=\lambda x+t_i\}_{i\in \I}$ be an equicontractive IFS and $b\geq 2$ be an integer. Suppose that $\frac{\log b}{\log |\lambda|}\notin\mathbb{Q}$. Then almost every $x$ is normal in base $b$ for any non-atomic self-similar measure of $\{\varphi_i\}_{i\in \I}$. 
\end{thm}
We remark that unlike in the statement of Conjecture \ref{Main conjecture}, we do not need the assumption that $\mu_{\p}$ be fully supported in the statement of Theorem \ref{Main theorem}. This is because even if $\mu_{\p}$ is not fully supported, it can still be realised as a fully supported self-similar measure for some other appropriate equicontractive IFS which satisfies $\frac{\log b}{\log |\lambda|}\notin\mathbb{Q}.$ The same is not true in the more general setting of Conjecture \ref{Main conjecture}. 

The following example gives one particular application of Theorem \ref{Main theorem}. We include it because it is of independent interest, and because it details a specific instance when Theorem \ref{Main theorem} applies to self-similar measures that are not Rajchman measures.

\begin{example}
	\label{Bernoulliconvolution}
Let $\lambda\in(1/2,1)$ and $\mu_{\lambda}$ be the distribution of the random sum $\sum_{n=0}^{\infty} \pm \lambda^{n}$ where plus and minus are chosen with equal probability. The probability measure $\mu_{\lambda}$ is known as the Bernoulli convolution. Bernoulli convolutions are a well studied family of measures. They have connections to the theory of algebraic numbers and to problems from Harmonic Analysis. Often we are interested in calculating the dimension of $\mu_{\lambda}$ and determining whether it is absolutely continuous. For a more detailed introduction to Bernoulli convolutions we refer the reader to the articles \cite{PSS,Sol2,Var2,Var} and the references therein. For our purposes, the important point is that $\mu_{\lambda}$ can be realised as the self-similar measure for the iterated function system $\{\varphi_1(x)=\lambda x -1 ,\varphi_{2}(x)=\lambda x +1\}$ and the probability vector $(1/2,1/2)$. 

In \cite{Erdos1} Erd\H{o}s proved that if $\lambda$ is the reciprocal of a Pisot number then $\mu_{\lambda}$ is not a Rajchman measure. This is significant for two reasons. First of all, it implies that $\mu_{\lambda}$ is singular with respect to the Lebesgue measure. Therefore normality results for $\mu_{\lambda}$ typical points cannot be immediately deduced from Borel's theorem. Secondly, it also means that we cannot use the work of Algom et al to establish normality results. As we will now explain, Theorem \ref{Main theorem} overcomes these obstacles and implies that if $\lambda$ is the reciprocal of a Pisot number and $b\geq 2$ is any integer, then $\mu_{\lambda}$ almost every $x$ is normal in base $b$. 

Let us fix $\lambda$ the reciprocal of a Pisot number and $b\geq 2$. By Theorem \ref{Main theorem}, to prove our statement it suffices to show that $\frac{\log b}{\log \lambda}\notin\mathbb{Q}$. For the purpose of obtaining a contradiction, let us suppose $\frac{\log b}{\log \lambda}\in\mathbb{Q}.$ This implies that $b^{-p/q}=\lambda$ for some $p,q\in\mathbb{N}$. Therefore for any $l\in\mathbb{N}$ we have that $\lambda^{-lq}\in\mathbb{N}$. Because $\lambda\in(1/2,1)$ we must have $\lambda^{-1}\in(1,2).$ Since $\lambda^{-1}$ is an algebraic integer, we may deduce that its minimal polynomial has degree at least $2$. This implies that $\lambda^{-1}$ has a non-empty set of Galois conjugates which we denote by $\{\gamma_1,\ldots,\gamma_{d}\}$. Using well known properties of algebraic integers, we know that for any $l\in\mathbb{N}$ we have 
\begin{equation}
\label{Pisot}
\lambda^{-lq}+\gamma_{1}^{lq}+\cdots+ \gamma_{d}^{lq}\in\mathbb{Z}.
\end{equation} Using the fact that each of the Galois conjugates has modulus strictly less than $1,$ it can be shown that there exists infinitely many $l$ for which $0<|\gamma_{1}^{lq}+\cdots+ \gamma_{d}^{lq}|<1.$ This fact together with \eqref{Pisot} contradicts the fact that $\lambda^{-lq}\in\mathbb{N}$ for all $l\in\mathbb{N}$. Therefore we must have $\frac{\log b}{\log \lambda}\notin\mathbb{Q}.$
\end{example}

We conclude this introductory section by surveying some other related works. In \cite{DGW} Dayan, Ganguly, and Weiss proved that if $\{\varphi_i(x)=\frac{x}{b} +t_i\}_{i\in \I}$ is an iterated function system and $t_i-t_{i'}\notin \mathbb{Q}$ for some $i,i'\in \I$, then $\mu_{\p}$ almost every $x$ is normal in base $b$ for every non-atomic fully supported self-similar measure. In \cite{Bak} the author studied powers of real numbers. They gave sufficient conditions for a self-similar measure to ensure that for $\mu_{\p}$ almost every $x$ the sequence $(x^n)_{n=1}^{\infty}$ is uniformly distributed modulo one. For an arbitrary Borel probability measure $\mu$, to prove that $\mu$ almost every $x$ is normal in base $b$ it is sufficient to prove that the Fourier transform of $\mu$ decays to zero sufficiently quickly. This fact follows from a result of Davenport, Erd\H{o}s, and LeVeque \cite{DEL}. The rate at which the Fourier transform of a fractal measure decays to zero is a well studied problem. For more on this topic we refer the reader to the papers \cite{ARHW,Bre,JorSah,Kau,LS,QR,SS,Sol,VY} and the references therein.

\section{Proof of Theorem \ref{Main theorem}}
For the rest of this paper we fix an equicontractive IFS $\{\varphi_i(x)=\lambda x +t_i\}_{i\in \I}$ and an integer $b\geq 2$ such that $\frac{\log b}{\log |\lambda|}\notin\mathbb{Q}.$ We now explain several assumptions that we can make without any loss of generality. These assumptions will help to simplify our proof. By considering the IFS $\{\varphi_{i}\circ \varphi_{i'}\}_{i,i'\in \I}$ if necessary, we can assume without loss of generality that $\lambda\in(0,1)$. Let us now also fix a non-atomic self-similar measure $\mu_{\p}$. As previously explained, we can assume without loss of generality that $\mu_{\p}$ is fully supported. It follows from the assumption that $\mu_{\p}$ is non-atomic, that for $M$ sufficiently large there exists  $(i_1,\ldots,i_M), (i_1',\ldots,i_M') \in \I^M$ such that 
\begin{equation}
\label{Memptyintersection}
\textrm{Conv}\left((\varphi_{i_1}\circ \cdots \circ \varphi_{i_M})(X)\right)\cap \textrm{Conv}((\varphi_{i_1'}\circ \cdots \circ \varphi_{i_M'})(X))=\emptyset.
\end{equation}
 Here and throughout $\textrm{Conv}(\cdot)$ is used to denote the convex hull of a set. Since each self-similar measure of $\{\varphi_i\}_{i\in \I}$ can be realised as a self-similar measure for $\{\varphi_{i_1}\circ \cdots \circ \varphi_{i_M}\}_{(i_1,\ldots,i_M)\in \I^M}$, it follows from \eqref{Memptyintersection} that without loss of generality we can assume that there exists $i',i''\in \I$ such that 
\begin{equation}
\label{emptyintersection}
\textrm{Conv}(\varphi_{i'}(X))\cap \textrm{Conv}(\varphi_{i''}(X))=\emptyset.
\end{equation} Last of all, since the property of being normal in base $b$ is preserved by integer translations and multiplication by $\frac{1}{b}$, without loss of generality we can assume that $X\subset [0,1)$. This final assumption will allows us to express our proof in terms of dynamics on the torus $\mathbb{R}/\mathbb{Z}$.

We now set out to prove that $\mu_{\p}$ almost every $x$ is normal in base $b$. Our proof is split into two parts. In the first part we show that it is possible to express $\mu_{\p}$ as the integral of some random probability measures. These random measures will resemble self-similar measures for iterated function systems satisfying the strong separation condition. This property means that blowing up small pieces of these measures can be interpreted dynamically in terms of the full shift on an appropriate sequence space. In the second part we use this observation together with a dynamical argument of Hochman to complete our proof.


\subsection{Disintegrating $\mu_{\p}$.}
Our method to disintegrate $\mu_{\p}$ is based upon a technique that first appeared in \cite{GSSY}, and was subsequently applied in \cite{KaeOrp} and \cite{SSS}. In these papers the authors used this technique to express an arbitrary self-similar measure as the integral of a collection of random measures. What was important for these authors was that these random measures could be expressed as an infinite convolution. Using the fact that our IFS is equicontractive, it can be shown that any of its self-similar measures automatically have this infinite convolution structure. For our purposes, the important difference is that use this technique to express $\mu_{\p}$ as the integral of a collection of random measures which each resembles a self-similar measure for an IFS satisfying the strong separation condition. 

It is useful at this point to establish some notation. In what follows we will let $\mu \ast \nu$ denote the convolution of two Borel probability measures on $\mathbb{R}$. Given $t\in \mathbb{R}$ we let $S_{t}:\mathbb{R}\to\mathbb{R}$ denote the map given by $S_{t}(x)=tx$. Moreover for $\mu$ a Borel probability measure we let $S_{t}\mu$ denote its pushforward under $S_{t}$.

Let $i',i''\in \I$ be as in \eqref{emptyintersection}. Let $$\Omega:=\left\{\{i',i''\}\right\}\bigcup_{\stackrel{i\in \I}{i\neq i',i\neq i''}}\left\{\{i\}\right\}.$$ We now define a probability vector $(q_{\omega})_{\omega\in \Omega}$ according to the rules $$q_{\omega}=p_{i'}+p_{i''} \textrm{ if }\omega=\{i',i''\}$$ and $$q_{\omega}=p_{i}\textrm{ if }\omega=\{i\} \textrm{ for }i \textrm{ such that }i\neq i'\textrm{ and }i\neq i''.$$ We denote the infinite product measure on $\Omega^{\mathbb{N}}$ corresponding to the probability vector $(q_{\omega})_{\omega\in \Omega}$ by $\mathbb{P}$.

Given $\omega\in \Omega$ we let $$[\omega]:=\left\{(\omega_n)_{n=0}^{\infty}\in \Omega^{\mathbb{N}}:\omega_0=\omega\right\}.$$ To any $\bm{\omega}=(\omega_n)_{n=0}^{\infty}\in \Omega^{\mathbb{N}}$ we associate the set $$X_{\om}:=\left\{\sum_{n=0}^{\infty}t_n\lambda^{n}:t_n\in \{t_i\}_{i\in\omega_n}\, \forall n\geq 0 \right\}.$$ Given $\om\in \Omega^{\mathbb{N}}$ and a finite word $(i_n)_{n=0}^{m}\in \I^{m+1}$ satisfying $i_{n}\in \omega_{n}$ for each $0\leq n\leq m,$ we associate the cylinder set $$X_{\om}\left((i_n)_{n=0}^{m})\right):=\left\{\sum_{n=0}^{\infty}t_n\lambda^{n}:t_n=t_{i_n} \textrm{ for }0\leq n\leq m \textrm{ and }t_n\in \{t_i\}_{i\in \omega_n}\, \forall n\geq m+1 \right\}.$$ Notice that for each $m\in\mathbb{N}$ we have the relation $$X_{\om}=\bigcup_{\stackrel{(i_n)_{n=0}^{m}\in \I^{m+1}}{i_n\in \omega_n\textrm{ for }0\leq n\leq m}}X_{\om}((i_n)_{n=0}^{m})).$$ Alternatively, we can rewrite this as 
\begin{equation}
\label{quasi self-similar}
X_{\om}=\bigcup_{\stackrel{(i_n)_{n=0}^{m}\in \I^{m+1}}{i_n\in \omega_n\textrm{ for }0\leq n\leq m}}(\varphi_{i_0}\circ \cdots \circ \varphi_{i_m})(X_{\sigma^{m+1}(\om)}).
\end{equation} Here $\sigma$ is the left shift on $\Omega^{\mathbb{N}}$. Equation \eqref{quasi self-similar} demonstrates that the set $X_{\om}$ satisfies a type of dynamical self-similarity analogous to \eqref{selfsimilarset}. We emphasise that the union in \eqref{quasi self-similar} is disjoint. This follows from a simple induction argument and the fact that any digit $\omega_n$ is either a single element set or $\omega_n=\{i',i''\}$ where $i'$ and $i''$ are as in \eqref{emptyintersection}.

To each $\om\in \Omega^{\mathbb{N}}$ we associate a probability measure $\mu_{\om}$ supported on $X_{\om}$ as follows 
$$\mu_{\om}:=\ast_{n=0}^{\infty}\sum_{i\in \omega_n}\frac{p_i}{q_{\omega_n}}\delta_{t_i\cdot \lambda^n}.$$ Alternatively, $\mu_{\om}$ can be interpreted as the law of the random sum $\sum_{n=0}^{\infty}t_n\lambda^{n}$ where for each $n$ the parameter $t_n$ is chosen from $\{t_i\}_{i\in\omega_n}$ with probabilities determined by the probability vector $(\frac{p_i}{q_{\omega_n}})_{i\in \omega_n}$. The following proposition describes the key properties of $\mu_{\om}$ that we will need in our proof.

\begin{prop}
\label{technical prop}
The following properties hold:
\begin{enumerate}
	\item $\mu_{\p}=\int \mu_{\om}\, d\mathbb{P}.$
	\item For $\mathbb{P}$ almost every $\om$ the measure $\mu_{\om}$ is non-atomic.
	\item For any $\om$ and finite word $(i_n)_{n=0}^{m}\in \I^{m+1}$ satisfying $i_n\in \omega_n$ for all $0\leq n\leq m$, we have $$\frac{\mu_{\om}|_{X_{\om}((i_n))}}{\mu_{\om}(X_{\om}((i_n)))}=(\varphi_{i_0}\circ \cdots \circ \varphi_{i_m})(\mu_{\sigma^{m+1}(\om)}).$$
\end{enumerate}
\end{prop}

\begin{proof}
We will prove each item in turn. The proof of item $1$ can be found in any of \cite{GSSY,KaeOrp,SSS}. We include the details for completion. We start by showing that $\int \mu_{\om}\, d\mathbb{P}$ satisfies the equation $\mu=\sum_{i\in \I}p_{i}\cdot \varphi_{i}\mu:$
\begin{align}
\label{Bernoullisplit}
\int \mu_{\om}\, d\mathbb{P}=\int \sum_{i\in \omega_0}\frac{p_i}{q_{\omega_0}}\cdot \delta_{t_i}\ast S_{\lambda}\mu_{\sigma(\om)}d\mathbb{P}&=\int \sum_{i\in \omega_0}\frac{p_i}{q_{\omega_0}}\cdot \varphi_{i}\mu_{\sigma(\om)}d\mathbb{P}\nonumber\\
&=\sum_{\omega_0\in\Omega}\int_{[\omega_0]}\sum_{i\in \omega_0}\frac{p_i}{q_{\omega_0}}\cdot \varphi_{i}\mu_{\sigma(\om)}d\mathbb{P}\nonumber\\
&=\sum_{\omega_0\in\Omega}\sum_{i\in \omega_0}\frac{p_i}{q_{\omega_0}}\int_{[\omega_0]} \varphi_{i}\mu_{\sigma(\om)}d\mathbb{P}\nonumber\\
&=\sum_{\omega_0\in\Omega}\sum_{i\in \omega_0}\frac{p_i}{q_{\omega_0}}\cdot \left(q_{\omega_0}\int \varphi_{i}\mu_{\om}d\mathbb{P}\right)\\
&=\sum_{\omega\in\Omega}\sum_{i\in \omega_0}p_i\int_{\Omega^{\mathbb{N}}} \varphi_{i}\mu_{\om}d\mathbb{P}\nonumber\\
&=\sum_{i\in \I}p_i \int \varphi_{i}\mu_{\om}d\mathbb{P}.\nonumber
\end{align}In line \eqref{Bernoullisplit} we used the fact that $\mathbb{P}$ is a product measure. We have shown that the probability measure $\int \mu_{\om}\, d\mathbb{P}$ satisfies the equation $\mu=\sum_{i\in \I}p_{i}\cdot \varphi_{i}\mu.$ The self-similar measure $\mu_{\p}$ is the unique probability measure satisfying this equation. Therefore $\mu_{\p}=\int \mu_{\om}\, d\mathbb{P}.$
\vspace{3mm}

We now move on to our proof of item $2$. We begin by remarking that for any $\om\in \Omega^{\mathbb{N}}$ and finite word $(i_n)_{n=0}^{m}$ satisfying $i_n\in \omega_n$ for $0\leq n\leq m,$ we have 
\begin{equation}
\label{nonatomic}
\mu_{\om}\left(X_{\om}((i_n)_{n=0}^{m})\right)\leq \left(\max\left\{\frac{p_{i'}}{p_{i'}+p_{i''}},\frac{p_{i''}}{p_{i'}+p_{i''}}\right\}\right)^{\#\{0\leq n\leq m:\omega_n=\{i',i''\}\}}.
\end{equation} For $\mathbb{P}$ almost every $\om$ the digit $\{i',i''\}$ occurs infinitely many times. Therefore the right hand side of \eqref{nonatomic} converges to $0$ as $m\to\infty$ for $\mathbb{P}$ almost every $\om$. Therefore for $\mathbb{P}$ almost every $\om,$ the $\mu_{\om}$ measure of a cylinder set $X_{\om}((i_n)_{n=0}^{m})$ converges uniformly to zero as $m\to\infty$. This implies that $\mu_{\om}$ is almost surely non-atomic.
\vspace{3mm}

We now focus on item $3.$ We remark that because the union in \eqref{quasi self-similar} is disjoint, the measure $\frac{\mu_{\om}|_{X_{\om}((i_n)_{n=0}^m)}}{\mu_{\om}(X_{\om}((i_n)_{n=0}^m))}$ is the law of the random sum $\sum_{n=0}^{\infty}t_n\lambda^n$ where $t_n=t_{i_n}$ for $0\leq n\leq m,$ and for $n\geq m+1$ each $t_n$ is chosen from $\{t_i\}_{i\in \omega_n}$ according to the law determined by the probability vector $(p_{i}/{q_{w_n}})_{i\in \omega_n}.$ This is precisely the pushforward of $\mu_{\sigma^{m+1}\om}$ by $\varphi_{i_0}\circ \cdots \circ \varphi_{i_m}$. This completes our proof.
\end{proof}

\subsection{Applying Hochman's argument}
To complete our proof of Theorem \ref{Main theorem} we apply an argument due to Hochman from \cite{Hoc}. Before recalling some results from this paper, it is useful to introduce some notation. For any $l\in\mathbb{R}$ we let $e_{l}(x)$ denote $e^{2\pi i lx}.$ Given a Borel probability measure $\mu$ we let $$\mathcal{F}_{l}(\mu):=\int e_{l}(x)\, d\mu(x).$$ We let $T_{b}:\mathbb{R}/\mathbb{Z}\to \mathbb{R}/\mathbb{Z}$ be given by $T_{b}(x)=bx\mod 1$. Given $(X,\mathcal{B},\mu)$ a probability space and $\A$ a measurable partition of $X$, for $x\in X$ we let $\A(x)$ denote the unique element of $\A$ containing $x$. Given $A\in \A$ for which $\mu(A)>0,$ we let $\mu_{A}$ denote the normalised restriction of $\mu$ to $A$, i.e. $\mu_{A}:=\frac{\mu|_{A}}{\mu(A)}.$ 

Let $$\theta:=-\frac{\log b}{\log \lambda}.$$ By our hypothesis we know that $\theta$ is irrational. We let $R_{\theta}:\mathbb{R}/\mathbb{Z}\to \mathbb{R}/\mathbb{Z}$ be given by $R_{\theta}(x)=x+\theta\mod 1.$ For each $n\in\mathbb{N}$ we let $$n'=\lfloor \theta n\rfloor.$$ Our parameter $n'$ has the property that $$b^n\lambda^{n'}=\lambda^{-\theta n}\lambda^{\lfloor \theta n\rfloor}=\lambda^{-R_{\theta}^n0}.$$


The following three statements are taken from \cite{Hoc}.

\begin{thm}{\cite[Theorem 2.2]{Hoc}}
	\label{Hoc1}
Let $T:X\to X$ be a continuous map of a compact metric space. Let $\A_1,\A_2,\A_3,\ldots$ be a refining sequence of Borel partitions. Let $\mu$ be a Borel probability measure on $X$ and assume that
$$\sup_{n\in\mathbb{N}}\{Diam (T^n(A)): A\in \A_{n+k},\mu(A)>0\}\to 0\qquad \textrm{ as } k\to\infty.$$ Then for $\mu$ almost every $x$, $$\lim_{N\to\infty}\left(\frac{1}{N}\sum_{n=1}^{N}\delta_{T^nx}-\frac{1}{N}\sum_{n=1}^{N}T^n\mu_{\A_{n}(x)}\right)=0$$ in the weak-* sense.
\end{thm}

\begin{thm}{\cite[Corollary 2.7]{Hoc}}
	\label{Hoc2}
Let $(X,\mu,T)$ be an ergodic measure preserving system on a compact metric space. Let $\beta>0$ and $\theta\neq 0$. Then for $\mu$ almost every $x$ the sequence $(n\theta,T^{\lfloor\beta n \rfloor}x)$ equidistributes for a measure $\nu_{x}$ on $[0,1)\times X$ that satisfies $\int \nu_{x}d\mu(x)=\tau \times \mu,$ where $\tau$ is the invariant measure on $(\mathbb{R}/\mathbb{Z},R_{\theta})$ supported on the orbit closure of $0$.
\end{thm}

\begin{lemma}{\cite[Lemma 3.2]{Hoc}}
	\label{Hoc3}
Let $\nu$ be a Borel probability measure on $\mathbb{R}$ and $\lambda\in(0,1)$. Then for every $r>0$ and $l\neq 0$,
$$\int_{0}^{1}|\mathcal{F}_{l}(S_{\lambda^{-t}}\nu)|^2\, dt\leq \frac{1}{r\cdot |l|\cdot \log \lambda^{-1}}+\int \nu(B_r(y))\,d\nu(y).$$
\end{lemma}


\noindent We now return to our proof. \begin{proof}[Proof of Theorem \ref{Main theorem}]
By Weyl's equidistribution criterion, to prove that $\mu_{\p}$ almost every $x$ is normal in base $b,$ it is sufficient to show that for any $l\in \mathbb{Z}\setminus\{0\}$ for $\mu_{\p}$ almost every $x$ we have \begin{equation}
\label{Weyl}
\lim_{N\to\infty}\frac{1}{N}\sum_{n=1}^{N}e_{l}(T_{b}^nx)=0.
\end{equation} Moreover, because of the disintegration $\mu_{\p}=\int \mu_{\om}\,d\mathbb{P}$ provided by item $1$ from Proposition \ref{technical prop}, it is in fact sufficient to show that for any $l \in \mathbb{Z}\setminus\{0\}$ we have that $\mu_{\om}$ almost every $x$ satisfies \eqref{Weyl} for $\mathbb{P}$ almost every $\om$. To establish this latter statement we will prove that for any $l\in\mathbb{Z}\setminus\{0\}$ and $\epsilon>0$ we have
\begin{equation}
\label{WTS}\mathbb{P}\left(\om: \mu_{\om} \textrm{ almost every } x \textrm{ satisfies }\limsup_{N\to\infty}\left|\frac{1}{N}\sum_{n=1}^{N}e_{l}(T_{b}^nx)\right|<\epsilon\right)>1-\epsilon.
\end{equation} To see why \eqref{WTS} implies this statement consider the set $$G:=\bigcap_{J=1}^{\infty}\bigcup_{j=J}^{\infty}\left\{\om:\mu_{\om} \textrm{ almost every } x \textrm{ satisfies }\limsup_{N\to\infty}\left|\frac{1}{N}\sum_{n=1}^{N}e_{l}(T_{b}^nx)\right|<\frac{1}{j} \right\}.$$ Equation \eqref{WTS} implies that $\mathbb{P}(G)=1.$ Moreover, for any $\om\in G$ we clearly have that $\mu_{\om}$ almost every $x$ satisfies \eqref{Weyl}. Therefore \eqref{WTS} implies our statement and so to complete our proof it is sufficient to show that \eqref{WTS} holds.

Let us now fix $l\in \mathbb{Z}\setminus \{0\}$ and $\epsilon>0$. We start our proof of \eqref{WTS} by stating the trivial fact that for any $x\in \mathbb{R}$ and $k\in\mathbb{N}$ we have
\begin{equation*}
\limsup_{N\to\infty}\left|\frac{1}{N}\sum_{n=1}^{N}e_{l}(T_{b}^nx)\right|=\limsup_{N\to\infty}\left|\frac{1}{N}\sum_{n=1}^{N}e_{l}(T_{b}^{n+k}x)\right|.
\end{equation*} Importantly the parameter $k$ here can be chosen to depend upon $\epsilon$. To prove that \eqref{WTS} holds we will eventually take $k$ to be sufficiently large in a way that depends upon $\epsilon$.

To each $\om\in \Omega^{\mathbb{N}}$ we associate the refining sequence of partitions $\A_1,\A_2,\A_3,\ldots,$ where for each $m$ the partition $\A_m$ is given by the cylinder sets corresponding to words of length $\lfloor \theta m\rfloor $, i.e. $\A_{m}=\{X_{\om}((i_n)_{n=0}^{\lfloor \theta m\rfloor-1}):i_n\in \omega_n \textrm{ for }0\leq n\leq \lfloor \theta m\rfloor-1\}.$ By Theorem \ref{Hoc1} it follows that for $\mu_{\om}$ almost every $x$ we have 
\begin{equation}
\label{orbit to measure}
\limsup_{N\to\infty}\left|\frac{1}{N}\sum_{n=1}^{N}e_{l}(T_{b}^{n+k}x)\right|=\limsup_{N\to\infty}\left|\frac{1}{N}\sum_{n=1}^{N}\F_{l}(T_{b}^{n+k}\mu_{\om,\A_{n}(x)})\right|.
\end{equation}By item 3. from Proposition \ref{technical prop}, we know that for any $x\in X_{\om}$ we have $$\mu_{\om,\A_{n}(x)}=\delta_{\sum_{j=0}^{n'-1}t_{i_j}\lambda^j}\ast S_{\lambda^{n'}}\mu_{\sigma^{n'}(\om)}$$ for some $(i_0,\ldots,i_{n'-1})\in \I^{n'}$. Therefore 
\begin{align*}
T_{b}^{n+k}\mu_{\om,\A_{n}(x)}=S_{b^{n+k}}\mu_{\om,\A_{n}(x)} \mod 1&= \delta_{b^{n+k}\cdot \sum_{j=0}^{n'-1}t_{i_j}\lambda^j}\ast S_{b^k}S_{b^n\lambda^{n'}}\mu_{\sigma^{n'}(\om)} \mod 1\\
&=\delta_{b^{n+k}\cdot \sum_{j=0}^{n'-1}t_{i_j}\lambda^j}\ast S_{b^k}S_{\lambda^{-R_{\theta}^n0}}\mu_{\sigma^{n'}(\om)} \mod 1\\
&=\delta_{b^{n+k}\cdot \sum_{j=0}^{n'-1}t_{i_j}\lambda^j}\ast S_{\lambda^{-R_{\theta}^n0}}S_{b^k}\mu_{\sigma^{n'}(\om)} \mod 1.
\end{align*}
Substituting the above into \eqref{orbit to measure}, we have the following for $\mu_{\om}$ almost every $x$:
\begin{align*}
\limsup_{N\to\infty}\left|\frac{1}{N}\sum_{n=1}^{N}e_{l}(T_{b}^{n+k}x)\right|&=\limsup_{N\to\infty}\left|\frac{1}{N}\sum_{n=1}^{N}\F_{l}\left(\delta_{b^{n+k}\cdot \sum_{j=0}^{n'-1}t_{i_j}\lambda^j}\ast S_{\lambda^{-R_{\theta}^n0}}S_{b^k}\mu_{\sigma^{n'}(\om)}\right)\right|\\
&\leq \limsup_{N\to\infty}\frac{1}{N}\sum_{n=1}^{N}\left|\F_{l}\left(\delta_{b^{n+k}\cdot \sum_{j=0}^{n'-1}t_{i_j}\lambda^j}\ast S_{\lambda^{-R_{\theta}^n0}}S_{b^k}\mu_{\sigma^{n'}(\om)}\right)\right|\\
&=\limsup_{N\to\infty}\frac{1}{N}\sum_{n=1}^{N}\left|\F_{l}\left( S_{\lambda^{-R_{\theta}^n0}}S_{b^k}\mu_{\sigma^{n'}(\om)}\right)\right|.
\end{align*} In the final line we have used the fact that for any Borel probability measure $\mu$ and Dirac mass $\delta_{y},$ we have $|\mathcal{F}_{l}(\delta_y\ast\mu)|=|\mathcal{F}_{l}(\mu)|$ for any $l\in\mathbb{Z}$. We emphasise that the last term only depends upon $\om$ and provides an upper bound for $$\limsup_{N\to\infty}\left|\frac{1}{N}\sum_{n=1}^{N}e_{l}(T_{b}^{n+k}x)\right|$$ for $\mu_{\om}$ almost every $x$. Applying Theorem \ref{Hoc2} we know that for $\mathbb{P}$ almost every $\om$ there exists a measure $\nu_{\om}$ on $\mathbb{R}/\mathbb{Z}\times \Omega^{\mathbb{N}}$ such that $$\lim_{N\to\infty}\frac{1}{N}\sum_{n=1}^{N}\left|\F_{l}\left( S_{\lambda^{-R_{\theta}^n0}}S_{b^k}\mu_{\sigma^{n'}(\om)}\right)\right|=\int|\mathcal{F}_{l}(S_{\lambda^{-t}}S_{b^k}\mu_{\om'})| d\nu_{\om}.$$ Moreover because $\theta$ is irrational, Theorem \ref{Hoc2} also implies that $\int \nu_{\om}\, d\mathbb{P}=\tau \times \mathbb{P}$ where $\tau$ is the Lebesgue measure. It follows from the above that to establish \eqref{WTS} it is sufficient to prove that 
\begin{equation}
\label{WTS2}
\mathbb{P}\left(\om: \int|\mathcal{F}_{l}(S_{\lambda^{-t}}S_{b^k}\mu_{\om'})| d\nu_{\om}<\epsilon\right)>1-\epsilon.
\end{equation} Using Markov's inequality, the relation $\int \nu_{\om}\, d\mathbb{P}=\tau \times \mathbb{P},$ and the Cauchy-Schwartz inequality, we have
\begin{align*}
\epsilon\cdot \mathbb{P}\left(\om:\int|\mathcal{F}_{l}(S_{\lambda^{-t}}S_{b^k}\mu_{\om'})| \nu_{\om}\geq \epsilon\right) &\leq \int \int |\mathcal{F}_{l}(S_{\lambda^{-t}}S_{b^k}\mu_{\om'})| d\nu_{\om}\,d\mathbb{P}\\
&=\int \int |\mathcal{F}_{l}(S_{\lambda^{-t}}S_{b^k}\mu_{\om})| d\tau\,d\mathbb{P}\\
&\leq \int \left(\int |\mathcal{F}_{l}(S_{\lambda^{-t}}S_{b^k}\mu_{\om})|^{2} d\tau\right)^{1/2} d\mathbb{P}.
\end{align*}Applying Lemma \ref{Hoc3} with $r=b^{k/2},$ it follows from the above that 
\begin{align*}
&\epsilon\cdot \mathbb{P}\left(\om:\int|\mathcal{F}_{l}(S_{\lambda^{-t}}S_{b^k}\mu_{\om'})| \nu_{\om}\geq \epsilon\right)\\
\leq &\int \left(\frac{1}{b^{k/2}\cdot |l| \cdot \log \lambda^{-1}}+\int S_{b^k}\mu_{\om}(B_{b^{k/2}}(y))\, dS_{b^k}\mu_{\om}\right)^{1/2}\, d\mathbb{P}\\
= & \int \left(\frac{1}{b^{k/2}\cdot |l| \cdot \log \lambda^{-1}}+\int \mu_{\om}(B_{b^{-k/2}}(y))\, d\mu_{\om}\right)^{1/2}\, d\mathbb{P}.
\end{align*} For an arbitrary non-atomic measure $\mu$ we know that $\lim_{k\to\infty}\int \mu(B_{b^{-k/2}}(y))\, d\mu=0.$ By item $2$ from Proposition \ref{technical prop} we know that for $\mathbb{P}$ almost every $\om$ the measure $\mu_{\om}$ is non-atomic. Therefore by choosing $k$ sufficiently large we have $$\int \left(\frac{1}{b^{k/2}\cdot |l| \cdot \log \lambda^{-1}}+\int \mu_{\om}(B_{b^{-k/2}}(y))\, d\mu_{\om}\right)^{1/2}\, d\mathbb{P}<\epsilon^2.$$ Applying this inequality in the above we may conclude that $$\mathbb{P}\left(\om:\int|\mathcal{F}_{l}(S_{\lambda^{-z}}S_{b^k}\mu_{\om'})| \nu_{\om}\geq \epsilon\right)<\epsilon.$$ This implies \eqref{WTS2} and completes our proof of Theorem \ref{Main theorem}. 
\end{proof}

\end{document}